\documentclass[letterpaper,11pt]{article}

\usepackage{amsthm}
\usepackage{amsmath}
\usepackage{amsfonts}
\usepackage{graphicx}

\usepackage{color}
\usepackage{xspace}
\newtheorem{thm}{Theorem}[section]

\newtheorem{lem}[thm]{Lemma}

\theoremstyle{definition}

\theoremstyle{remark}
\newtheorem{rem}[thm]{Remark}
\numberwithin{equation}{section}

\def\bx{{\bf x}}
\def\by{{\bf y}}
\def\balpha{{\bf \alpha}}
\def\bbeta{{\bf \beta}}

\def\R{\mathbb R}
\def\E{\mathbb E}
\def\eps{\varepsilon}
\def\ind{\operatorname{ind}}
\def\Id{\operatorname{Id}}
\def\med{M}

\def\net{{\cal N}}

\newcommand{\card}[1]{\operatorname{card} #1}
\newcommand{\K}[1]{}  
\newcommand{\noK}[1]{#1}  
\def\polylog{\operatorname{polylog}}
\def\C{{\mathbb C}}
\def\ln{\operatorname{log}}
\def\log{\operatorname{log}}
\newcommand{\hr}[1]{}
\newcommand{\na}[1]{}

\def\tp{^*\xspace}
\title{Fast and RIP-optimal transforms}

\author{Nir Ailon\thanks{N.\ Ailon is with the Technion Israel Institute of Technology}, Holger Rauhut\thanks{H.\ Rauhut is with the Hausdorff Center for Mathematics and
the Institute for Numerical Simulation, University of Bonn, Endenicher Allee 60, 53115 Bonn, Germany, \rm{rauhut@hcm.uni-bonn.de}}}
\begin{document}

\maketitle
\begin{abstract}
We study constructions of  
$k \times n$ matrices $A$ that both (1) satisfy the restricted isometry property (RIP) 
at sparsity $s$ with optimal parameters, and  (2) are efficient in the sense that only $O(n\log n)$ operations are required to
compute $Ax$ given a vector $x$.  Our construction is based on repeated application of independent transformations of the form $DH$, where $H$ is a Hadamard or Fourier transform and $D$ is a 
diagonal matrix with random $\{+1,-1\}$ elements on the diagonal, followed by 
any 
$k \times n$ matrix of orthonormal rows
(e.g.\ selection of $k$ coordinates).  We provide guarantees (1) and (2) for 
a larger regime of parameters for which such constructions were previously unknown.  Additionally, our construction
does not suffer from the extra poly-logarithmic factor multiplying
the number of observations $k$ as a function of the sparsity $s$, as present in the currently best known RIP estimates 
for partial random Fourier matrices and other classes of structured random matrices.
\end{abstract}


\section{Introduction}

The theory of compressive sensing predicts that sparse vectors can be stably reconstructed from a small number of linear measurements
via efficient reconstruction algorithms including $\ell_1$-minimization \cite{do06-2,fora11}. The restricted isometry property (RIP) of the measurement matrix streamlines 
the analysis of various reconstruction algorithms \cite{cata06,carota06-1,fo10,fora13}. All
known matrices that satisfy the RIP in the optimal parameter regime (see below for details) are based on randomness. Well-known
examples include Gaussian and Bernoulli matrices where all entries are independent. Unfortunately, such matrices do not 
possess any structure and therefore no fast matrix-vector multiplication algorithm. 
The latter is important for speed-up of recovery algorithms. This article addresses constructions of
matrices that satisfy the RIP in the optimal parameter regime and have fast matrix-vector multiplication algorithms.

A vector $x \in \C^n$ is said to be $s$-sparse if the number of nonzero entries of $x$ is at most $s$.
A matrix $A \in \C^{k \times n}$ satisfies the RIP
with respect to parameters $(s,\delta)$ if, for all $s$-sparse vectors $x\in \C^n$,
\begin{equation}\label{basic} 
(1-\delta)\|x\|_2 \leq \|Ax\|_2 \leq (1+\delta)\|x\|_2,
\end{equation}
where $\|\cdot\|_2$ denotes the Euclidean norm.\footnote{In much of the related literature, the definition of RIP uses \emph{squared} Euclidean norms.  The definition (\ref{basic}) is, however, more convenient for our purposes. Of course, both versions are equivalent up to a transformation
of the parameter $\delta$.}
\na{(I moved the footnote from inside the sentence to after the period - I think it's the correct way to do so in English, at least that's what some anoynous pedantic reviewer told me once)}
If $A$ satisfies the RIP with parameters $(2s, \delta^* )$ for a suitable $\delta^* < 1$ then a variety of recovery algorithms reconstruct 
an $s$-sparse vector exactly from $y = A x$. Moreover, reconstruction is stable under passing from sparse to approximately sparse vectors
and under adding noise on the measurements. The value of $\delta\tp $ depends 
only on the reconstruction algorithm \cite{carota06-1,netr08,blda09,fo10,fora13}. 

It is well-known by now \cite{badadewa08,cata06,mepato09} that a 
Gaussian random matrix (having independent normal distributed entries of variance $1/m$)
satisfies the RIP with parameters $(s,\delta)$ with probability at least $1-e^{-c \delta^2 k}$ if
\[
k \geq C \delta^{-2} s \ln(n/s), 
\]
where $c,C>0$ are universal constants. Using lower bounds for Gelfand widths  of $\ell_p$-balls for $0 < p \leq 1$, 
it can be shown that $k$ must be at least $C_{\delta} s\log(n/s)$ for  
the RIP to hold \cite{FoucartPRU10}. 
It can further be shown \cite{fora13} that the constant (as a function of $\delta$) satisfies
$C_{\delta} \geq C \delta^{-2}$.
Since we will always assume in this paper that 
$s \leq Cn^{1/2}$, $\log(n/s)$ is equivalent to $\log(n)$ up to constants.
Hence, we will say that a $k \times n$ matrix is {\bf RIP-optimal} at $s$ 
if it satisfies the RIP with $(s,\delta)$ for 
\[
\delta = C \sqrt{\frac k {s\log n}}\ .
\]
 (The reader should keep in mind that for large $s$,
RIP optimality should be defined to hold when $k$ is at most $C\delta^{-2}s\log(n/s)$.)

The restricted isometry property is closely connected to Johnson-Linden\-strauss embeddings.
We say that a random $k \times n$ matrix $A$ satisfies the Johnson-Lindenstrauss  property (JLP) with parameters 
$(N,\delta)$ if for any set $X \subseteq \R^n$ of cardinality at most 
$N$  (\ref{basic}) holds uniformly for all $x\in X$ with constant probability.
It is well known that a matrix of independently drawn subgaussian elements satisfies JLP if $k \geq C\delta^{-2}\log N$.
Specializations of this fact 
to Gaussians and to Bernoulli random variables can be found in \cite{Indyk:1998:ANN:276698.276876,DBLP:journals/jcss/Achlioptas03}. The general claim is  obtainable by noting that the subgaussian
property is the crux of these proofs.  If $A$ satisfies JLP with $(N,\delta)$ for $k \leq C\delta^{-2}\log N$, then we say that $A$ is JLP-optimal.

The JLP and RIP properties are known to be almost equivalent, in a certain sense.   
One direction is stated as follows: A (random) $k \times n$  matrix satisfying JLP 
with $(N,\delta)$ satisfies
RIP with $(\hr{c}C(\log N)/(\log n), \delta)$ \na{(agreed)} with constant probability \cite{badadewa08,mepato09}.   
Hence, for any arbitrarily small $\mu>0$ an RIP-optimal matrix can be obtained by drawing a  JLP-optimal matrix with $k=C \delta^{-2}s \log n$.
The derivation of RIP from JLP is a specialization of JLP to a set $X$ consisting of a $\eps$-net of $s$ sparse unit vectors, for $\eps = 0.1$ say, which has cardinality
$N \leq (C n)^s$.

The other direction is a remarkable recent result by Krahmer and Ward \cite{DBLP:journals/siamma/KrahmerW11} implying that if $A$ has RIP with $(s,\delta/4)$, 
then $AD$ has JLP with $(N, \delta)$ as long as $N \leq 2^s$, where $D$ is a diagonal matrix with independent random signs ($\pm 1$) on the diagonal.  Notice that from this
result, RIP-optimality of $A$ does not imply JLP-optimality of $AD$, because RIP-optimality implies
that the embedding dimension of $k$ is at least $C\delta^{-2} s\log n$, which suffers from an additional factor of $\log n$ compared to the JLP-optimality guarantee bound
of $C\delta^{-2}\log N = C\delta^{-2}s$  (for $N=2^s$).  From this observation we intuitively conclude that RIP-optimality is weaker than JLP-optimality, 
and hence expect that RIP-optimal constructions should be easier to obtain.  The main results of this paper, roughly speaking, confirm this by providing constructions of RIP-optimal matrices 
which are simpler than previously known constructions\na{\cite{DBLP:journals/dcg/AilonL09,aich09}} \hr{WHICH ONES?}that relied on JLP optimality.

\subsection{Known Constructions of RIP-optimal and JLP-optimal matrices}

No deterministic RIP-optimal matrix constructions are known.  Deterministic constructions of RIP matrices are known only for a grossly suboptimal regime of parameters.
See \cite{DBLP:journals/siamma/KrahmerW11} for a nice survey of such constructions.

Of particular interest are RIP or JLP matrices $A$ that are \emph{efficient} in the sense that for any vector $x\in \C^n$, $Ax$ can be computed in time $O(n\log n)$. 
Such constructions are known for JLP-optimal (and hence also RIP-optimal) matrices as long as $k \leq n^{1/2-\mu}$ for any arbitrarily small $\mu$ \cite{DBLP:journals/dcg/AilonL09}.
This is achieved with the transformation $BHD^{(1)}HD^{(2)}H\cdots HD^{(r)}$, 
where $D^{(i)}$ are independent random sign diagonal matrices, $H$ are Hadamard transforms
and $B$ is a subsampled (and rescaled) Hadamard transform, where the subset 
of sampled coordinates is related to a carefully constructed dual binary code, and $r$ is at most $C/\mu$.
For larger $k$,
the best efficient constructions satisfying RIP are due to Rudelson and Vershynin \cite{ruve08} 
with $k \geq Cs\log^4 n$, namely  a factor of $\log^3 n$ away from optimal, see also \cite{ra10}. (Note that
at least two of the $\log n$ factors can be improved to $\log s$, but in light of 
the aforementioned positive results, we can assume that $s$ is at least, say, $n^{1/3}$.)
This was recently improved by Nelson et al. to only $\log^2 n$ factors away from optimal \cite{NelsonPW12}.
The construction in \cite{ruve08} is a Fourier transform 
followed by a subsampling 
operator.  
Another family of RIP almost optimal matrices with a fast transform algorithm is that of partial random circulant matrices and time-frequency structured random matrices.
The best known results in that vein have recently appeared in the work of Krahmer, Mendelson and Rauhut \cite{KrahmerMR12}, where RIP matrices of almost optimal embedding
dimension $k= Cs\log^4 n$ are designed. 
These constructions improve on previous work in \cite{rarotr12}.

\subsection{Contribution}
In this work we construct RIP-optimal, efficient matrices for the regime $s\leq n^{1/2-\mu}$ that are simpler than those implied from the JLP construction in \cite{DBLP:journals/dcg/AilonL09}.
For $s\leq C n^{1/3}/\log^{2/3} n$, we show that the transformation $SHDHD'H$ is RIP-optimal, where $H$ is a Hadamard or Fourier transform, $D$ and $D'$ are independent  random sign diagonal matrices
and $S$ is an \emph{arbitrary} deterministic subsampling (and rescaling) matrix of order $k$.  This complements a previous RIP result in a similar parameter regime of $s \leq n^{1/3}/\polylog(n)$ \hr{(Should we be more specific here? Danger of confusion between the two regimes just mentioned!)} \na{done}
implied by a JLP construction by Ailon and Chazelle in \cite{aich09}, 
in which the transformation $THD$ is used, where $T$ is a sparse random matrix with $k^3$ nonzero elements.
 The surprising part about our construction is that
$S$ can be taken to 
be an arbitrary (deterministic) subsampling matrix.  The random subsampling used in the construction \cite{aich09} is, in a sense, traded off with two more Fourier transforms and one more 
random diagonal matrix here.
For $s\leq Cn^{1/2-\mu}$, we show that the transformation $PD^{(1)}HD^{(2)}H\cdots D^{(r)} H$
is RIP-optimal and efficient, where the $D^{(i)}$'s are as above, $P$ is an \emph{arbitrary} deterministic matrix with properly normalized, pairwise orthogonal rows and $r$ is at most $C/\mu$.
This is simpler than the  RIP-optimal efficient construction implied by the aforementioned  JLP-optimal efficient construction in \cite{DBLP:journals/dcg/AilonL09}, because no binary code designs are necessary.

Our main proof techniques involve concentration inequalities of vector valued Rademacher chaos of degree $2$. For the second construction an additional bootstrapping argument is applied 
that, roughly speaking, shows that the RIP parameters of $ADHD'H$ are better than those of $A$.

\section{Notation and Main Results}

Throughout, the letter $C$ denotes a general global constant, whose value may change from appearance to appearance.
The integer $n$ denotes the ambient dimension, $k \leq n$ denotes the embedding dimension, and 
$\C^n$ denotes the $n$ dimensional
complex space with standard inner product. The usual $\ell_p$-norms are denoted by $\|\cdot\|_p$, the spectral norm of
a matrix $A$ by $\|A\|$ and the Frobenius norm as $\|A\|_F = \sqrt{\operatorname{trace}(A\tp A)}$.

\noindent
We let $H\in \C^{n\times n}$ denote a fixed matrix with the following properties:
\begin{enumerate}
\item $H$ is unitary,
\item the maximal absolute entry of $H$ is $\K{K}n^{-1/2}$\K{, for some number $K>0$ which we assume fixed},
\item the transformation $Hx$ given a vector $x\in \C^n$ can be computed in time $O(n\log n)$.
\end{enumerate}
 \K{The discrete Fourier matrix  is an example of such matrix (with $K=1$), and the discrete cosine transform is another (with $K=\sqrt 2$).   
The Walsh-Hadamard matrix (for $n=2^\ell$) is another such (real valued) example.}
\noK{Both the discrete Fourier matrix and the Walsh-Hadamard matrix are examples of such matrices.}
\noK{Note that the upper bound of $n^{-1/2}$ in property 2. above could be replaced by $Kn^{-1/2}$ for any constant
$K$, thus encompassing transformations such as the discrete cosine transform (with $K=\sqrt 2$) with little
affect on the guarantees.  We have
decided to concentrate on the case $K=1$ for simplicity.}

For any vector $z\in \C^n$, $D_z$ denotes a diagonal
matrix with the elements of $z$ on the diagonal.
For a subset $\Omega$ of $\{1,\dots, n\}$, let  $\ind(\Omega)\in \R^n$ denote the vector with $1$ at coordinates $i\in \Omega$ and $0$ elsewhere.  Then define $P_\Omega = D_{\ind(\Omega)}$ and
$R_\Omega \in \R^{k \times n}$ to be the map that restricts a vector in $\R^n$ to its entries in $\Omega$, i.e., a subsampling operator.

Recall that a $k \times n$ matrix $A$ has the RIP property with respect to parameters $(s,\delta)$ where $s$ is an integer and $\delta>0$, if for any $s$-sparse unit vector $x\in \C^n$,
$$ 1-\delta  \leq \|Ax\|_2 \leq 1+\delta\ .$$
(Note that we allow $\delta>1$, unlike typical definitions of RIP).
For a fixed sparsity parameter $s$, we denote by $\delta_s(A)$
the infimum over all $\delta$ such that $A$ has the RIP property
with respect to parameters $(s,\delta)$.   We say that $A$ 
is RIP optimal at a given sparsity parameter $s$ if 
\begin{equation}\label{defopt} \delta_s(A) \leq C \sqrt{\frac{s\log n} k}\ .
\end{equation} 

A random vector $\epsilon$ with independent entries that take the values $\pm 1$ with equal probability is called a Rademacher vector.

Our first main result provides a simple RIP-optimal matrix with a fast transform algorithm for small sparsities $s = {\cal O}(n^{1/3}/\log^{2/3} n)$.

\begin{thm}\label{thm1}
Let $\epsilon, \epsilon'\in \{\pm 1\}^n$ be two independent Rademacher vectors,
and let $\Omega$ be any subset of $\{1,\dots, n\}$ of size $k$.
The $k \times n$ random  matrix $A = R_\Omega HD_{\epsilon}H D_{\epsilon'} H$
is RIP-optimal
for the regime  $k \leq \sqrt{n/s}$. More precisely, if 
\begin{equation}\label{cond:k}
\sqrt{n/s} \geq k \geq C\K{K} \delta^{-2} s \ln n,
\end{equation} 
then $A$ satisfies the RIP \eqref{basic} with probability at least $1-e^{-C\delta^2 k}$.
In particular, the conditions on $k$ entail
\[
s \leq \frac{C\delta^{4/3} n^{1/3}}{\K{K^{2/3}}\ln^{2/3} n}\ .
\]
\end{thm}
Clearly, $A = R_\Omega HD_{\epsilon}H D_{\epsilon'} H x$ can be computed in ${\cal O}(n \log n)$ time by assumption on $H$.
Also note that \eqref{cond:k} implies a restriction on $\delta$ for which the result applies, namely
\begin{equation}\label{thm1:rest:delta}
\delta \geq c\frac{s^3 \log n}{n}.
\end{equation}

Our second main result gives a RIP-optimal matrix construction with a fast 
transform for the enlarged parameter regime $s = O(\sqrt{n}/\log n)$ and $k = O(n/(s\log n))$.

\begin{thm}\label{thm2}
Assume $s_0 \leq s \log n \leq k \leq \sqrt n$ and 
$\kappa = C\sqrt{\frac{sk\log n} n} < 1/2$, where $s_0$ is a global constant.
Let $A$ be an arbitrary $k \times n$  matrix satisfying $AA\tp  = \frac n k \Id_k$.
Let $r = \left  \lceil \frac{ -\log(2\sqrt{n/k})}{\log \kappa}\right \rceil$, and 
let $\epsilon_{(1)},\dots, \epsilon_{(r)} \in \{\pm 1\}^n$ denote independent Rademacher vectors.
Then the $k \times n$ matrix 
\begin{equation}
\hat A = A D_{\epsilon_{(1)}}HD_{\epsilon'_{(1)}} H D_{\epsilon_{(2)}}HD_{\epsilon'_{(2)}} H
\cdots
D_{\epsilon_{(r+1)}}HD_{\epsilon'_{(r+1)}} H
\end{equation}
is RIP-optimal with probability at least $0.99$, that is, \eqref{basic} holds if $k \geq C \K{K} \delta^{-2} s \log n$.

In particular, if we strengthen the constraints by requiring $s \leq n^{1/2-\mu}$ for some global $\mu>0$, 
then $\hat A$ is also computationally efficient in the sense that $\hat A x$ can be computed in time ${\cal O}(n \log n)$.
\end{thm}

The probability $0.99$ in the above theorem is arbitrary and can be replaced by any different value in $(0,1)$. This effects only
the constant $C$. However, we remark that the present proof does not seem to give the optimal dependence of $C$ 
in terms of the probability bound.

It is presently not clear whether the restrictions $s = {\cal O}(n^{1/3}/\K{(K^{2/3}}\ln^{2/3} n\K{)})$ and $s = {\cal O}(\sqrt{n}/\log n)$ in the above theorems can be removed.
In any case, regimes of small $s$ are the ones of most interest in compressive sensing, anyway.

Section~\ref{cbrtn} is dedicated to proving the Theorem~\ref{thm1}, and Section~\ref{sqrtn} proves Theorem~\ref{thm2}.

\section{The regime $s = {\cal O}(n^{1/3}/\K{(K^{2/3}}\log^{2/3} n\K{)})$} 
\label{cbrtn}

Our first randomized, computationally efficient, RIP-optimal construction involves three
applications of $H$, two random sign diagonal matrices, and a choice of an \emph{arbitrary}
set of $k$ coordinates.
 Fix  $x\in U_s := \{ x \in \C^n: \|x\|_2 = 1, x \mbox{ is } s\mbox{-sparse} \}$
 and let $\epsilon, \epsilon' \in \{\pm 1\}^n$ be two random sign vectors.   
Consider the following random variable, indexed by $x$,
$$ \alpha(x) = \sqrt{\frac n k}\| P_\Omega HD_\epsilon HD_{\epsilon'} Hx\|_2\ ,$$
where $k$ is the cardinality of $\Omega$.
\noindent
It is not hard to see that $\E[\alpha(x)^2] = 1$.  Indeed, denoting $\tilde x = HD_{\epsilon' }Hx$ and conditioning on a fixed value of $\epsilon'$,  for any $i\in \{1,\dots n\}$,

$$ \E_\epsilon \|P_{\{i\}} H D_\epsilon \tilde x\|^2_2 = \|\tilde x\|_2^2/n = 1/n\ .$$

The random variable $\alpha(x)$ is the norm of a decoupled Rademacher chaos of
degree $2$.  For the sake of notational convenience, we denote, for $i,j\in \{1,\dots n\}$,
\begin{equation}\label{def:xij}
 \bx_{ij} = \sqrt {\frac n k} P_\Omega H P_{\{i\}} H P_{\{j\}} H x\ ,
 \end{equation}
so that we can conveniently write 
\[
\alpha(x) =  \|\sum_{i=1}^n\sum_{j=1}^n \epsilon_i\epsilon'_j \bx_{ij}\|_2.
\]
By a seminal result of Talagrand \cite{Tal96}, a Rademacher chaos concentrates around its median. 
We will exploit the following version.
\begin{thm} With a double sequence $\bx_{ij}$, $i,j=1,\hdots n$, of vectors in $\C^n$ and two independent Rademacher
vectors $\epsilon, \epsilon' \in \{\pm 1\}^n$ let 
\[
\alpha = \| \sum_{i=1}^n\sum_{j=1}^n \epsilon_i\epsilon'_j \bx_{ij}\|_2.
\] 
Let $\med_{\alpha}$ be a median of $\alpha$. For 
$\by \in \C^n$ introduce the $n \times n$ matrix $B_\by = (\by\tp  \bx_{ij})_{i,j=1}^n$ and
the parameters 
\begin{align}
U & = \sup_{\by \in \C^n, \|\by\|_2 \leq 1} \|B_{\by}\|\label{defU}\\
V & = \E \sup_{\by \in \C^n, \|\by\|_2 \leq 1} (\|B_{\by} \epsilon\|_2^2 + \|B_{\by}\tp  \epsilon'\|_2^2)^{1/2}.\label{defV}
\end{align}
Then, for $t > 0$,
\begin{equation}\label{tal:chaos2}
\Pr(| \alpha - \med_{\alpha} | \geq t ) \leq 2 \exp\left(-C \min\left\{\frac{t^2}{V^2},\frac{t}{U}\right\} \right)\ .
\end{equation}
\end{thm}
\begin{proof} With
$A = (\bx_{ij})_{i,j=1}^n$ and 
\[
S = \frac{1}{2} \left(\begin{matrix} 0 & A \\ A\tp  & 0\end{matrix}\right), \qquad \tilde{\epsilon} = \left(\begin{matrix} \epsilon \\ \epsilon' \end{matrix} \right)
\]
we can rewrite the decoupled chaos as the coupled symmetric chaos
\[
\tilde{\epsilon}\tp  S \tilde{\epsilon} = \sum_{i,j=1}^n \epsilon_i \epsilon_j' \bx_{ij},
\]
where matrix multiplication is extended in an obvious way to matrices with vector-valued entries. Observe that $S$ has zero diagonal.
Therefore, the claim follows from Theorem~1.2  in \cite{Tal96}.
\end{proof}

We bound the quantity $U=U(x)$ in \eqref{defU}, 
where the $\bx_{ij}$ are defined by \eqref{def:xij}. 
Note that $\sum_{i,j} \balpha_i\bbeta_j \by\tp  \bx_{ij} = \by\tp  HD_{\balpha} H D_{\bbeta} H x$, and hence
\begin{eqnarray}
U & = & \sup_{\|\by\|_2, \|\balpha\|_2, \|\bbeta\|_2\leq 1}  \sum_{i=1}^n \sum_{j=1}^n \bar \balpha_i \bbeta_j \by\tp \bx_{ij} 
= \sqrt{\frac n k} \sup_{\by, \balpha, \bbeta} \by\tp  P_\Omega H D_{\balpha\tp} HD_{\bbeta} H x  \nonumber \\
&=& \sqrt{\frac n k}\sup_{\by, \balpha, \bbeta}  \balpha\tp  D_{\by\tp  P_\Omega H} H D_{Hx} \bbeta  \nonumber \\
&\leq&  \sqrt{\frac n k} \sup_{\by, \balpha, \bbeta}\|\balpha \|_2 \cdot  \|D_{\by\tp  P_\Omega
 H}\| \cdot \|D_{Hx}\| \cdot \|\bbeta\|_2 \nonumber\\
&\leq&  \sqrt{\frac n k}  \sup_{\by, \bbeta} \|\by\tp  P_\Omega H\|_{\infty} \cdot \|Hx\|_\infty \cdot \|\bbeta\|_2 \nonumber\\
&\leq&  \sqrt{\frac n k} \sup_{\by} \|\by\tp  P_{\Omega}\|_1 \K{K}n^{-1/2} \cdot \|x\|_1 \K{K}n^{-1/2} 
\leq  \sqrt{\frac n k} k^{1/2} \K{K^2} n^{-1/2} \|x\|_1 n^{-1/2}\nonumber\\ 
&\leq& \K{K^2} \sqrt {s/n} \ . \label{Ubound}
\end{eqnarray}

\noindent
To bound $V$, we define a process $\nu(\by)$  as
\begin{equation}\label{process} \nu(\by) = \sqrt{ \|B_{\by} \epsilon\|_2^2 + \|B_{\by}\tp  \epsilon'\|_2^2}\,
\end{equation}
so that $V = \E \sup_{\|\by\|\leq 1} \nu(\by)$.
By the definition of the vectors $\bx_{ij}$, it clearly suffices to take the supremum on vectors $\by$ supported on $\Omega$.
For any such $\by$, let $\mu_\by = \E \nu(\by)$.
Jensen's inequality yields
\begin{equation}\label{boundexpectation111}
\mu_{\nu(\by)} \leq \sqrt{\E \nu_\by^2} \leq \sqrt {\frac {2n} k} \| D_{\by\tp  P_\Omega H} H D_{Hx} \|_F\ .
\end{equation}
By definition of the Frobenius norm, together with the fact that the matrix elements of $H$ are bounded above by $\K{K}n^{-1/2}$ in absolute
value, we obtain
\begin{equation}\label{boundexpectation}
\mu_{\nu(\by)} \leq \K{K} \|\by\|\sqrt{2/k}\ .
\end{equation}
We use a concentration bound for vector-valued Rademacher sums 
(tail inequality (1.9) in~\cite{LedouxT10}) to 
notice that for any $\by$ and $t>0$,
\begin{equation}\label{tal} \Pr\left(| \nu(\by)-M_{\nu(\by)}|  > t\right) \leq 4\exp\left(-Ct^2/ \sigma_{\nu(\by)}^2 \right)\ ,\end{equation}
\hr{(constant 2 replaced by 4?)}\na{(You're right - fixed)}
where $M_{\nu(\by)}$ is a median of $\nu(\by)$ and, with $A= \sqrt{n/k} P_\Omega$,
\begin{eqnarray}
\sigma_{\nu(\by)} &=& \sup_{\|\beta\|^2_2+\|\gamma\|^2_2\leq 1}   \left (\sum_{j=1}^n\left |\sum_{i=1}^n \beta_i \by\tp  \bx_{ij} + \gamma_i \by\tp \bx_{ji}\right|^2\right)^{1/2}  \nonumber \\ 
&\leq& \sup_{\|\beta\|,\|\gamma\|\leq 1} \|  \by\tp  A D_\beta H D_{Hx} \|_2 + \|  D_{\by\tp  A}HD_{\gamma} Hx \|_2  \nonumber \\
&\leq& \sup_{\stackrel{\|\beta\|,\|\gamma\| \leq 1}{\|\beta'\|,\| \gamma'\| \leq 1}} \by\tp  A D_\beta H D_{\beta'}Hx +  \by\tp  AD_{ \gamma}HD_{\gamma'} Hx  \nonumber \\
&\leq& 2\sup_{\|\beta\|,\|\gamma\|\leq 1} \by\tp  A D_\beta H D_{\gamma}Hx   \nonumber \\
&\leq &  2\sqrt{n/k} \|D_{\by\tp  P_\Omega H}H D_{Hx}\|   \label{onebeforelast111}  \\
& \leq  & 2\|\by\|_2\cdot \|x\|_1 \K{K}/\sqrt n \leq  2\K{K}\sqrt{s/n}\ . \label{boundsigma}
\end{eqnarray}
Upper bounding the expression (\ref{onebeforelast111}) was done exactly as above when upper bounding $U$.
Using  (\ref{boundexpectation}) and the second part of Lemma~\ref{integrate}, we conclude that
\begin{equation}\label{boundMmu}
M_{\nu(\by)} \leq \mu_{\nu(\by)} + C \sigma_\by \leq \K{K}\|\by\|\sqrt{2/k} + C\K{K}\sqrt{s/n}\ .
\end{equation}
We will now bound $V$.  
 To that end, we use 
a general epsilon-net argument. Given a subset $T$ of a Euclidean space, we recall that 
a set $\net \subset T$ is called  $\mu$-separated if $\|\by - \by'\|_2 > \mu$ for all $\by, \by' \in \net$, $\by \neq \by'$.
It is called maximally $\mu$-separated if no additional vector can be added to $\net$ in a $\mu$-separated position.

\begin{lem}\label{baraniuk}
Let $\gamma : \C^m \mapsto \R^+$ be a seminorm, and let $\net$ denote a maximal $\mu$-separated set of Euclidean
unit vectors $\by \in \C^m$ for some $\mu<1$.  
Let
$$S = \sup_{\by\in \net} \gamma(\by)\;\;\; I = \inf_{\by\in\net}\gamma(\by)\ .$$
Then
\begin{eqnarray}\label{net1}
\sup_{\|\by\|=1} \gamma(\by) &\leq& \frac 1{1-\mu} S \label{net1up}\\
\inf_{\|\by\|=1} \gamma(\by) &\geq& I - \frac{\mu}{1-\mu} S \ . \label{net1down}
\end{eqnarray}
In particular, if $\kappa := \sup_{\by\in \net} |\gamma(\by)-1|$, then
\begin{equation}\label{net2}
\sup_{\|\by\|=1} |\gamma(\by)-1| \leq \frac {\kappa+\mu} {1-\mu}\ .
\end{equation}
\end{lem}
The proof of the bound \eqref{net1up} is contained in \cite{ve12}. The proof of \eqref{net1down} is similar and implicitly
contained in \cite{badadewa08}.

We use the lemma by constructing a maximal $\eta = 0.1$-separated set $\net$ of 
$S_\Omega := \{ \by \in \C^n : \|\by\|_2 = 1, \operatorname{supp} \by \subset \Omega\}$. 
Using a standard
volumetric argument (see e.g.\ \cite[Proposition 10.1]{ra10})
\begin{equation}\label{card}
 \card{\net} \leq (1+2/\eta)^{2k}  
 = 21^{2k}\ .
\end{equation}
We also notice that $\nu(\cdot)$ is a seminorm \na{(indeed a seminorm)} for any fixed $\epsilon, \epsilon'$. 
Using (\ref{net1up}),
$$ \sup_{\stackrel{\by: \by=P_\Omega \by}{ \|\by\|=1}} \nu(\by) \leq \frac 1{1-0.1} \sup_{\by'\in\net} \nu(\by') \  . $$
Taking expectations yields
\begin{equation}\label{boundsups} \E \sup_{\stackrel{\by: \by=P_\Omega \by}{ \|\by\|=1}} \nu(\by) \leq 1.2 \E \sup_{\by'\in\net} \nu(\by') \  . 
\end{equation}
\noindent
The expectation on the right hand side can now be bounded, in light of (\ref{tal}) and using
Lemma~\ref{expmaximaeq} (with $\sigma_i' = 0$), as follows:
\begin{equation}\label{comparegaus2}
 \E  \sup_{\by'\in\net} \nu(\by')  \leq \sup_{\by}M_{\nu(\by)} + \sup_{\by} \sigma_{\nu(\by)}\sqrt{\log \card{\net}}\ .
\end{equation}
Together with (\ref{boundMmu}),  (\ref{card}), and (\ref{boundsigma}), this implies
\begin{equation}\label{comparegaus}
 \E  \sup_{\by'\in\net} \nu(\by')    \leq  C\K{K}\left(\sqrt{1/k} + \sqrt{s/n} + \sqrt{\frac s n k}\right) \ .
\end{equation}
If we now assume that $k \leq \frac 1 2 \sqrt{\frac {n}{s}}$, 
then we conclude from (\ref{boundsups}), 
(\ref{comparegaus}) and  (\ref{Ubound}) that
$$ V=  \E \sup_{\stackrel{\by: \by=P_\Omega \by}{ \|\by\|=1}} \nu(\by) \leq C\K{K}/\sqrt k \quad \mbox{ and } \quad  U \leq 2\K{K^2}/k\ .  $$
\hr{(Why the index 4 at $C_4$ here? Should we use a `generic' constant $C$ everywhere or maybe number through all constants?)}
\na{(A painful issue indeed - either the constant enumeration should be organized, or $C$ should be used everywhere - let's do that after the arxiv)}

\noindent
Plugging these upper bounds into (\ref{tal:chaos2}) we 
conclude  that for all $0 < t\leq  1$
$$ \Pr(|\alpha(x) - \med_{\alpha(x)}| \geq t) \leq 2\exp\left( - Ct^2 k\K{/K^2}\right)\ $$
(because $\min\{t^2/V^2,t/U\} = Ct^2/V^2$ for $0 <t\leq 1$ and for the derived values of $U,V$).
But we also know, using Lemma~\ref{integrate}, that
$ | \sqrt{\E \alpha(x)^2} - M_{\alpha(x)} | \leq C\K{K}/\sqrt k$.  Recalling that $\E \alpha(x)^2 = 1$ and
combining, we conclude that  for $t\leq 1$,
\begin{equation}\label{tail111} \Pr(|\alpha(x) - 1| \geq t) \leq C\exp\left\{ - C t^2k\K{/K^2}\right\}\ .\end{equation}

Now we fix a support set $T \subset \{1,\hdots, N\}$ of size $s$ and consider the complex unit sphere restricted to $T$, i.e.,
$S_T = \{x \in \C^n: \|x\|_2 = 1, \operatorname{supp} x \in T\}$. Let $\net_T$ be a maximal $\eta$-separated set of $S_T$, which has
cardinality at most $(1+2/\eta)^{2s}$ by \eqref{card}. 
By a union bound, we have
\begin{align}
\Pr(\max_{x \in \net_T} |\alpha(x)-1| \geq t) & \leq (1+2/\eta)^{2s} C e^{-C t^2 k} \notag\\
& = C \exp\left( -C t^2 k + 2s \ln(1+2/\eta)\right)\ . \notag
\end{align}
It follows from Lemma \ref{baraniuk} that 
\begin{align}
&\Pr(\max_{x \in U_s} |\alpha(x)-1| \geq (t+\eta)/(1-\eta))\notag\\
& = \Pr\left(\max_{\#T = s} \max_{ x \in S_T} |\alpha(x) - 1| \geq (t+\eta)/(1-\eta)\right)\notag\\
& \leq \sum_{\#T = s} \Pr\left(\max_{x \in \net_T} |\alpha(x) - 1| \geq t\right) 
\leq \left( \begin{matrix} n \\ s \end{matrix}\right)  C \exp\left( -C t^2 k + 2s \ln(1+2/\eta)\right)\notag\\
&\leq \exp\left(-C t^2 k + 2s \ln(1+2/\eta) + s \ln(en/s)\right) \ .\notag
\end{align}
Choosing $\eta = \min\{t,0.5)$ we conclude that $\max_{x \in U_s} |\alpha(x)-1| \leq 4t$ with probability at least $1-\varepsilon$ if
\[
k \geq Ct^2(s (\ln(1+2/t) + \ln(en/s)) + \ln(C \varepsilon^{-1}) \ .
\]
Replacing $t$ by $\delta/4$ and noting that \eqref{thm1:rest:delta} implies $\ln(1+8/\delta) \leq c \ln(en/s)$ concludes
the proof.

\section{The regime $s={\cal O}(\sqrt{n}/\log n)$}
\label{sqrtn}

We now use the idea developed in the previous section to bootstrap an efficient RIP-optimal construction for  a larger regime.

Let $A$ be a fixed $k\times n$ matrix with pairwise orthogonal rows of Euclidean length $\sqrt{n/k}$ each.  Namely, 
\begin{equation}\label{assumptionA}
A A\tp  = \frac n k \Id_k\ .
\end{equation}
\hr{$A^* $ or $A^T $ throughout the paper?}
\na{It's a macro \texttt{$\backslash$tp} now}

The strategy will be to improve the RIP parameter $\delta_s(A)$ of $A$ by replacing $A$ with $\tilde A = A D_\epsilon H D_{\epsilon'} H$.
To analyze the (random) RIP parameter $\delta_s(\tilde A)$, fix an $s$-sparse unit vector $x \in \C^n$.  Now define the random variable
$$ \alpha(x) = \|\tilde A x\|_2 .$$

As before, we note that $\alpha(x)$ is the norm of a decoupled Rademacher chaos of degree $2$ 
in a $k$-dimensional Hilbert space,
which can be conveniently written as $\alpha(x) = \|\sum_{i=1}^n\sum_{j=1}^n \epsilon_i\epsilon'j \bx_{ij}\|_2$, where
$$ \bx_{ij} = AHP_{\{i\}} H P_{\{j\}} H x\ .$$

As in the previous discussion, we bound the invariants of interest $U$ and $V$ as defined in (\ref{defU}) and (\ref{defV}), respectively.  We start by bounding $U$.  By definition,
\begin{eqnarray}\label{boundU2}
U &\leq&  \sup_{\stackrel{\by, \balpha, \bbeta}{\|\by\|_2,\|\balpha\|_2,\|\bbeta\|_2\leq1}} \by\tp  A D_\balpha HD_\bbeta H x\ .
\end{eqnarray}
Notice now that since $x$ is $s$-sparse by assumption, we have $\|x\|_1\leq \sqrt s$ and hence $\|Hx\|_\infty \leq \K{K} \sqrt{s/n}$ so that
$$ \|D_\bbeta Hx\|_2 \leq \K{K} \sqrt{s/n} \quad \mbox{ and } \quad  \|D_\bbeta Hx\|_1 \leq 1\ . $$
The right hand side inequality is due to Cauchy-Schwarz.
In turn, this implies
$$ \|HD_\bbeta Hx\|_\infty \leq 1/\sqrt n \quad \mbox{ and } \quad  \|HD_\bbeta Hx\|_2 \leq \K{K}\sqrt{s/n}\ .$$
\noindent
Therefore, 
\begin{equation}\label{L1L2}\|D_\balpha HD_\bbeta Hx\|_2 \leq 1/\sqrt n \K{\leq K/\sqrt n}\quad \mbox{ and } \quad \|D_\balpha H D_\bbeta Hx\|_1 \leq \K{K} \sqrt{s/n}\ .
\end{equation}
Again, the right hand side inequality is due to Cauchy-Schwarz.  We need the following simple lemma, see also \cite[Lemma 3.1]{plve11}.

\begin{lem}
Let $w\in \C^n$ be such that $\|w\|_1\leq \sqrt s \rho$ and $\|w\|\leq \rho$ for some
integer $s$ and number $\rho>0$.  Then there exist $N=\lceil n/s\rceil$ vectors $w^{(1)},\dots, w^{(N)}$ such that $w^{(i)}$ is $s$-sparse for each $i$, $w=\sum_{i=1}^N w^{(i)}$ and $\sum_{i=1}^N \|w^{(i)}\|_2 \leq 2\rho$.
\end{lem}
\begin{proof}
Assume wlog that the coordinates $w_1,\dots, w_n$ of $w$ are sorted
so that $|w_1| \geq |w_2| \geq \cdots \geq |w_n|$.
 For $i=1,\dots, N$, let  $$w^{(i)} = (\underbrace{0,\dots, 0}_{(i-1)s\mbox{ times}}, w_{(i-1)s+1},\dots, w_{is}, 0, \dots, 0)\tp  \in \C^n$$ and $\alpha_i = \|w^{(i)}\|_\infty = |w_{(i-1)s+1}|$.
Clearly $w=\sum_{i=1}^N w^{(i)}$, and we have:
\begin{eqnarray}
\sum \|w^{(i)}\| &=& \|w^{(1)}\| + \sum_{i=2}^N \|w^{(i)}\| 
\leq \rho + \sum_{i=2}^N \alpha_i \sqrt s\ . \label{eq:12norm}
\end{eqnarray}
But now notice that for all $i=1,\dots, N$, $\|w^{(i)}\|_1 \geq \alpha_{i+1}s$ (where we define $\alpha_{N+1}=0$), hence
we conclude, using the assumptions, that
$$ \sum_{i=2}^N \alpha_i s \leq \sum_{i=1}^N \|w^{(i)}\|_1  = \|w\|_1 \leq \sqrt s \rho .$$
\noindent
Therefore, $\sum_{i=1}^N \alpha_i \sqrt s \leq \rho$.
Together with (\ref{eq:12norm}), this implies the lemma.
\end{proof}
\begin{rem} Note that the technique of grouping together monotonically decreasing coordinates of a vector in blocks is 
rather standard in compressive sensing,
see for example \cite{carota06-1}  or \cite{DBLP:journals/siamma/KrahmerW11}.
\end{rem}
From (\ref{L1L2}) we conclude that $D_\balpha HD_\bbeta H x$ can be decomposed as $\sum_{i=1}^N w^{(i)}$ as in the lemma
with $\rho = \K{K}n^{-1/2}$.  Using the RIP assumption on $A$, for each $i=1,\dots, N$, $\|Aw^{(i)}\| \leq \|w^{(i)}\|(1+\delta)$.
By the lemma's premise, and using the triangle inequality, this implies
\begin{equation}\label{boundU3} U \leq 2(1+\delta)\K{K}/\sqrt n\ .
\end{equation}

\noindent
Bounding $V$ is done as follows.  We define the process $\nu$ as 
$$ 
\nu(\by) = 
\left(\|B_{\by} \epsilon\|_2^2 + \|B_{\by}\tp  \epsilon'\|_2^2\right)^{1/2}\ ,
$$
where $(B_{\by})_{ij} = \by\tp  \bx_{ij}$,
over the set $\{\by \in \R^k: \|\by\|\leq 1\}$, so that $V = \E \sup_{\by} \nu(\by)$.
For any $\by$, $\nu(\by)$ is a Rademacher sum
in $k$-dimensional Hilbert space.  Thus, we can use \eqref{tal} 
to conclude that for all $\by$,
$$ 
\Pr\left(\nu(\by) > M_{\nu(\by)} + t\right) \leq 4\exp(-t^2/8\sigma_{\nu(\by)}^2)\ ,
$$
where $M_{\nu(\by)}$ is a median of $\nu(\by)$ and 
\begin{eqnarray}
\sigma_{\nu(\by)} &=& \sup_{\|\beta\|^2+\|\gamma\|^2\leq 1}   \left (\sum_{j=1}^n\left |\sum_{i=1}^n \beta_i \by\tp  \bx_{ij} + \gamma_i \by\tp \bx_{ji}\right|^2\right)^{1/2}  \nonumber \\ 
&\leq& \sup_{\|\beta\|,\|\gamma\|\leq 1} \|  \by\tp  AD_\beta H D_{Hx} \| + \|  D_{\by\tp  A}HD_{\gamma} Hx \|  \nonumber \\
&\leq& \sup_{\stackrel{\|\beta\|,\|\gamma\| \leq 1}{\|\beta'\|,\| \gamma'\| \leq 1}} \by\tp  AD_\beta H D_{\beta'}Hx +  \by\tp  AD_{ \gamma}HD_{\gamma'} Hx  \nonumber \\
&\leq& 2\sup_{\|\beta\|,\|\gamma\|\leq 1} \by\tp  AD_\beta H D_{\gamma}Hx \label{onebeforelast} \\
&\leq & 4(1+\delta)\K{K}/\sqrt n\ . \label{boundsigma5}
\end{eqnarray}
For the last inequality, notice that (\ref{onebeforelast}) is bounded by twice the RHS of (\ref{boundU2}), and recall the derivation of
(\ref{boundU3}).
Using the first part of Lemma~\ref{integrate}, we conclude that for all $\by$ such that $\|\by\|=1$,
\begin{equation}\label{boundMnuy}
M_{\nu(\by)} \leq \mu_{\nu(\by)} + C\K{K}(1+\delta)/\sqrt n\ ,
\end{equation} 
where $\mu_{\nu(\by)}$ is the expectation of $\nu(\by)$.  
Jensen's inequality yields
\begin{align}
\mu_{\nu(\by)} &\leq \|D_{H\tp A\tp \by} H D_{Hx}\|_F \leq \|H\tp A\tp \by\| \cdot \|Hx\|\K{K}/\sqrt n
= \|A\tp \by \| \cdot \|x\| \K{K}/\sqrt n \notag\\
& \leq (\sqrt{n/k})\K{K}/\sqrt n = \K{K}k^{-1/2}\ .\notag
\end{align}
Again we notice that for any fixed $\epsilon, \epsilon'$, $\nu$ is a seminorm.
As before, let $\net$ denote a maximal $0.1$-separated set of Euclidean unit vectors in $\C^k$.
Hence by (\ref{net1up}) in  Lemma~\ref{baraniuk}, for any fixed $\epsilon, \epsilon'$,
$$ \sup_{\|\by\|=1} \nu(\by) \leq 1.2 \sup_{\by\in \net} \nu(\by)\ .$$
Taking expectation on both sides and using Lemma~\ref{expmaximaeq} to bound the right hand side (recalling that the
cardinality of $\net$ is at most $21^{2k}$), we conclude
$$ 
V = \E \sup_{\|\by\| = 1} \nu(\by) \leq \sup_{\|\by\|=1} M_{\nu(\by)} +  C\sqrt{k}\sup_{\|\nu(\by)\|=1} \sigma_{\nu(\by)} .$$
By (\ref{boundsigma5}) and by our bound on $M_{\nu(\by)}$,
we conclude
\begin{eqnarray}
V &\leq& \K{K} {k^{-1/2}} + C\K{K}(1+\delta)/\sqrt n + C\K{K} \sqrt{k}(1+\delta)/\sqrt n \nonumber \\ 
&\leq&  \K{K} ({k^{-1/2}}+ (1+\delta)C\sqrt {k/n})\ . \label{boundV3}
\end{eqnarray}
From (\ref{tal:chaos2}), (\ref{boundU3}), (\ref{boundV3}) we then conclude that for all $t>0$, 
\begin{align}
& \Pr(|\alpha(x) - \med_{\alpha(x)}| \geq t) \notag\\
& \leq 2\exp\left( - C \min\left\{ \frac{t^2}{(\K{K}(1+\delta)\sqrt{k/n}+1/\sqrt k)^2}, \frac {t\sqrt n}{\K{K}(1+\delta)}\right\}\right)\ .\label{tal3}
\end{align}
Using the first part of Lemma~\ref{integrate} and recalling that $\E \alpha(x)^2 = 1$ this implies that
\begin{equation}\label{median_var}
\left | M_{\alpha(x)} - 1\right | \leq C\K{K}((1+\delta)\sqrt{ k/n}+1/\sqrt k)\ ,
\end{equation}

We now use the net-technique to pass to the supremum over all $s$-sparse unit vectors to provide
an estimate of the restricted isometry constant.
For each subset $T \subset \{1,\hdots,n\}$ of cardinality $s$ we consider a maximal $\mu$-separated 
$\net_T$ of the unit sphere $S_T$ of complex
unit length vectors with support $T$ where $\mu = 1/k$. By \eqref{card} and since $k < \sqrt{n/s}$ and $s \leq \sqrt{n}$, 
the union $\net = \cup_{\#T = s} \net_T$ is bounded in size
by
\begin{equation}
\# \net \leq \left(\begin{matrix} N \\ s\end{matrix} \right) (1+2/\eta)^{2s} \leq (en/s)^{s} (1+2\sqrt{n/s})^{2s}
\leq \exp(C s \log n) \ . \label{card:bound}
\end{equation}
Using Lemma~\ref{expmaximaeq}, (\ref{tal3}) and (\ref{median_var}), we conclude that
\begin{eqnarray}
\E \sup_{x\in \net} |\alpha(x) -1 | &\leq& C\K{K}((1+\delta)\sqrt{k/n}+1/\sqrt k) 
+ C\K{K}\sqrt{s\log n}(1+\delta)\sqrt{k/n}  \nonumber\\ 
& & + C\K{K} (s\log n)(1+\delta)/\sqrt n \nonumber \\
& \leq & C\K{K}\left((1+\delta)\left (\sqrt{\frac{ s k \log n}{ n}} + \frac {s\log n}{\sqrt n}\right ) + k^{-1/2}\right) \ .\label{boundalpha1_} 
\end{eqnarray}
We will assume in what follows that
\begin{equation}\label{assume7}
s\log n \leq k\ ,
\end{equation}
so that (\ref{boundalpha1_}) takes the simpler form
\begin{eqnarray}
\E \sup_{x\in \net} |\alpha(x) -1 |
& \leq & C\K{K}\left ((1+\delta)\sqrt \frac {sk\log n}{ n} +k^{-1/2}\right)\ .\ \label{boundalpha1} 
\end{eqnarray}

Recalling that $\alpha$ is a seminorm and applying (\ref{net2}) in Lemma~\ref{baraniuk} 
we pass to the set of all $s$-sparse Euclidean unit normed vectors,
\begin{align}
\E \sup_{\stackrel{\|x\|=1}{\|x\|_0\leq s}} |\alpha(x) -1 | & = \E \max_{\# T = s} \sup_{x \in S_T}  |\alpha(x) -1 |\notag\\
& \leq  C\K{K}\left ((1+\delta)\sqrt \frac {sk\log n}{ n} +k^{-1/2}\right)\ .\ \label{boundalpha11} 
\end{align}

\subsection{A Bootstrapping Argument}

Let $\delta'$ denote $\delta_s(\tilde A)$.
Assume henceforth that the parameters $s,k$ satisfy
\begin{equation}\label{sparsityassump}
\kappa := C \K{K}\sqrt{ \frac {sk\log n}{ n}}  < 1/2\ .
\end{equation}
Clearly (\ref{boundalpha1}) is
 a bound on  $\E [\delta']$.
With the new notation, we get
$$ \E [1+\delta'] \leq (1+\delta) \kappa + 1 +  C\K{K}k^{-1/2}\ .
$$
Denote $\tilde A$ by $A^{(1)}$ and $A$ by $A^{(0)}$.  Now consider inductively repeating the above process, obtaining  (for $i\geq 2$) $A^{(i)}$ from $A^{(i-1)}$ by $$A^{(i)} = A^{(i-1)}  D_{\epsilon_{(i)}} H D_{\epsilon'_{(i)}} H\ ,$$ where $\epsilon_{(i)}, \epsilon'_{(i)}$ are independent copies of $\epsilon, \epsilon'$.  Let $\delta^{(i)}$ denote $\delta_s(A^{(i)})$. By independence and 
the principle of conditional expectation, we conclude that
$$ \E[1+\delta^{(i)}] \leq (1+\delta^{(0)})\kappa^i +  \frac{1 + C\K{K}k^{-1/2}}{(1-\kappa)}\leq (1+\delta^{(0)})\kappa^i + 2(1 + C\K{K}k^{-1/2})\ .$$

Assume in what follows that $k$ is large enough so that $C\K{K}k^{-1/2} \leq 1$. Then last inequality conveniently implies $\E[1+\delta^{(i)}] \leq (1+\delta^{(0)})\kappa^i + 4$.
Recall by our definition of $A$  that $\delta=\delta^{(0)}$ can be no more than $\sqrt{n/k}$.
Let $r$ be taken  as
\begin{equation}\label{defi1}
r := \left  \lceil \frac{ -\log(2\sqrt{n/k})}{\log \kappa}\right \rceil 
\end{equation}
so that $(1+\delta(0))\kappa^{r} \leq (1+\sqrt{n/k})\kappa^{r} \leq 2\sqrt{n/k}\kappa^{r} \leq 1$ and hence
$$ \E[1+\delta^{(r)}] \leq 5\ .$$

Using Markov inequality, this implies that with probability
at least, say,  $0.995$
\begin{equation}\label{event} 1+\delta^{(r)} \leq  1000\ .\end{equation}
From now on assume event (\ref{event}) holds.  
Now for an $s$-sparse unit vector $x$, let $x^{(r+1)} := A^{(r+1)} x$.
The assume $k \leq \sqrt{n}$ is equivalent to $1/\sqrt k \geq \sqrt{k/n}$.  Using this, and substituting a constant for $(1+\delta)$, (\ref{tal3}) implies 
\begin{align}
& \Pr(| \|x^{(r+1)}\|_2 - \med_{\|x^{(r+1)}\|_2}| \geq t) \notag\\
 & \leq 2\exp\left( - C \min\left( \frac{t^2}{(\K{K}k^{-1/2})^2 }, t\K{K}\sqrt n\right)\right)\ . \label{tal4}
\end{align}
where $M_{\|x^{(r+1)}\|_2}$ is a median of $\|x^{(r+1)}\|_2$.

Once again we consider maximal $\mu$-separated sets $\net_T$ of $S_T$ with $\mu = 1/k$ 
for each $T \subset \{1,\hdots,N\}$ of size $s$
and form $\net = \cup_{\#T = s} \net_T$.
The cardinality of $\net$ is at most $\exp\{Cs\log n\}$, see \eqref{card:bound}.
We can now use a union bound over $\net$, to conclude that with probability at least $0.995$,
\begin{equation}\label{i1plus1_}
\max_{x\in \net} \left | \|x^{(r+1)}\|_2 - \med_{\|x^{(r+1)}\|_2} \right | \leq  C\K{K}\max\left(\sqrt{\frac{s \log n}{k}},\  \frac{s\log n}{\sqrt n}\right)\ .
\end{equation}
Using (\ref{median_var}), this implies
\begin{eqnarray}
\max_{x\in \net} \left | \|x^{(r+1)}\|_2 - 1 \right |  &\leq& C\K{K}\left( \max\left(\sqrt{\frac{sk \log n}{n}},\  \frac{ s\log n}{\sqrt n}\right) +k^{-1/2}\right) \notag \\
&\leq &  C\K{K}\left(\sqrt{\frac{sk \log n}{n}} + k^{-1/2}\right)\ ,\label{i1plus1} 
\end{eqnarray}
where the last inequality used (\ref{assume7}). As before,
using (\ref{net2}) in Lemma~\ref{baraniuk}, allows us to pass to the set of all $s$-sparse vectors:
\begin{eqnarray}
\sup_{\stackrel{\|x\|=1}{\|x\|_0\leq s}} \left | \|x^{(r+1)}\|_2 - 1 \right |  &\leq& 
  C\K{K}\left(\sqrt{\frac{sk \log n}{n}} + k^{-1/2}\right)\ .\label{i1plus111} 
\end{eqnarray}
\noindent
Recalling the assumption $k \leq \sqrt n$,  this implies
\begin{equation*}
\delta^{(r+1)}  \leq C\K{K}\sqrt{\frac{s\log n} {k}}\ .
\end{equation*}
and the proof of Theorem~\ref{thm2} is concluded.

\appendix

\section{ Properties of Mixed Gaussian and Exponential Processes}\label{app}
\begin{lem}\label{integrate}
Assume $X$ is a  random variable 
such  that for some number $M$ and for all $t\geq 0$,
$$ \Pr[|X-M| > t] \leq C\exp\left \{-\min\{t^2/\sigma_1^2, t/\sigma_2\}\right \}\ ,$$
for some $\sigma_1,\sigma_2 \geq 0$.   Then
\begin{enumerate}
\item $\left |(\E X^2)^{1/2} - |M|\right| \leq C'' \sqrt{\sigma_1^2 + \sigma_2^2} $
\item 
$ |\E X - M| \leq C''  \left |\sigma_1+\sigma_2\right |$

\end{enumerate}
for some constant $C''$ that depends only on $C$.
\end{lem}
\begin{proof}
For the first part, assume that $M\geq 0$ for the moment.
By integrating and changing variables,
\begin{eqnarray}
\E(X-M)^2 &=& \int_{0}^\infty \Pr[|X-M| \geq \sqrt t] dt  \nonumber \\
&\leq& C\int_{0}^\infty \exp\left \{-t/\sigma_1^2\right\} dt + C\int_0^\infty \exp \left \{-\sqrt t/\sigma_2\right\} dt  \nonumber\\
&=& C\sigma_1^2\int_{0}^\infty e^{-s}ds + 2C\sigma_2^2\int_{0}^\infty e^{-s}sds \nonumber \\
&\leq& (\sigma_1^2+\sigma_2^2) C' \label{appendix_ineq}
\end{eqnarray}
for some constant $C'>0$.
On the other hand, with $p^2 = \E X^2$ we have $\E(X-M)^2 = p^2 + M^2 - 2M\E X$ and $\E X \leq \E |X| \leq \sqrt{ \E X^2} = p$.  We hence conclude that $(p-M)^2 \leq (\sigma_1^2+\sigma_2^2)C'$.
If $M < 0$ then we simply replace the random variable $X$ with $-X$ and
$M$ with $-M$.

The second part is obtained in the same way by integrating to bound $\E |X-M|$.

\end{proof}

\def\Cmxma{C_{\operatorname{mx}}}
\begin{lem}\label{expmaxima}
Assume that $X_i$, $i=1,\hdots,N$ are random variables 
such that for each $i$ there exist numbers $M_i$ and $\sigma_i,\sigma'_i\geq 0$ such that for all $t\geq 0$,
$$ \Pr[|X_i-M_i| > t] \leq 2\exp\left \{-\min\{t^2/\sigma_i^2, t/\sigma'_i\}\right \}\ .$$
Then
\begin{equation}\label{expmaximaeq}
\E \sup_i|X_i-M_i| \leq C\left( \sqrt{\log N}\sup_i \sigma_i + \log N \sup_i\sigma'_i\right ) \ .
\end{equation}
\end{lem}
Note that the variables $X_i$ are not required to be independent.
The proof can be done by integration by parts, very similar to the derivation of (3.6) in \cite{LedouxT10}.

\section*{Acknowledgements}

Nir Ailon acknowledges the support of a Marie Curie International Reintegration Grant PIRG07-GA-2010-268403, and a grant from the GIF, the
German-Israeli Foundation for Scientific Research and Development.  Work was done while his visiting the Hausdorff Center for Mathematics
at the University of Bonn.
 Holger Rauhut acknowledges support by the Hausdorff Center for Mathematics at the University of Bonn and 
funding by the European Research Council through the grant StG 258926.


\end{document}